\documentclass[ 5p]{elsarticle}

\usepackage{hyperref}

\journal{System and Control Letters}

\usepackage{amssymb}
\usepackage{amsthm}
\usepackage{color}
\usepackage{caption}
\usepackage{subcaption}

\usepackage{mathtools}
\mathtoolsset{showonlyrefs=true}

\usepackage{amsfonts}
\usepackage{graphicx}%
\providecommand{\U}[1]{\protect\rule{.1in}{.1in}}
\newtheorem{theorem}{Theorem}[section]

\newtheorem{corollary}[theorem]{Corollary}

\newtheorem{definition}[theorem]{Definition}

\newtheorem{proposition}[theorem]{Proposition}

\newcommand{\e}{\varepsilon}

\newcommand{\R}{\mathbb{R}}

\begin{document}

\begin{frontmatter}

\title{A model for system uncertainty in reinforcement learning}

\author{Ryan Murray and Michele Palladino}
\address{Department of Mathematics, The Pennsylvania State University, University Park, PA, USA}

%
%

\begin{abstract}
This work provides a rigorous framework for studying continuous time control problems in uncertain environments. The framework considered models uncertainty in state dynamics as a measure on the space of functions. This measure is considered to change over time as agents learn their environment. This model can be seem as a variant of either Bayesian reinforcement learning or adaptive control. We study necessary conditions for locally optimal trajectories within this model, in particular deriving an appropriate dynamic programming principle and Hamilton-Jacobi equations.  This model provides one possible framework for studying the tradeoff between exploration and exploitation in reinforcement learning.
\end{abstract}


\end{frontmatter}


\section{Introduction}

Recently a lot of attention in the machine learning community has been given to methods for reinforcement learning. This has been rewarded with significant advances in machine learning, e.g. the recent development of computer algorithms to beat human Go players \cite{AlphaGo}. Reinforcement learning can be seen as an extension of classical adaptive control methods \cite{sutton1992}. Roughly speaking, reinforcement learning seeks to solve optimal control problems with limited information about state dynamics and objective values. This article aims to propose and study an optimal control model which is closely related to many problems typical to reinforcement learning.

A common setting for reinforcement learning is the following: one considers a discrete state space, with some Markov (possibly stochastic) transitions between these states, and where the movement from one state to another is affected by a control (these are called \emph{Markov decision processes}). Popular algorithms from reinforcement learning solve this type of problem by iteratively estimating a value function using the dynamic programming principle, and then recovering the optimal control by using the value function (optimal synthesis of the  feedback control). This is known as the \emph{value iteration} algorithm in reinforcement learning. Although many other algorithms, such as policy iteration, Q-learning, temporal difference and policy gradient methods, can also be used, they all rely on similar underlying frameworks. An excellent introduction to the field can be found in \cite{BartoSuttonBook}.

At this point we make a few observations about the reinforcement learning framework. First, the discrete framework, which is very natural to the computer science community, is not very convenient for understanding underlying structure of these systems. For example, the discrete framework is not amenable to characterizing necessary or sufficient conditions, or to understanding realistic convergence rates. Of course the discrete framework is useful theoretically (as one has compactness for free), but the convergence guarantees tend to depend poorly on the number of states (which is overly pessimistic when considering problems with underlying continuum structure). Some excellent works have focused on moving to continuum reinforcement learning problems \cite{Munos,Doya}. These works are naturally focused on algorithmic concerns (i.e. finding appropriate function bases), and less on proving properties about such models.

Second, in the framework of reinforcement learning, very little is assumed about state dynamics or objectives. Some algorithms conduct a \emph{model free} approach, which does not seek to construct a model for underlying state dynamics. Other flavors of the algorithms attempt to model the underlying dynamics of the system; this is known as \emph{model-based reinforcement learning}. A mathematically clear exposition of these two frameworks can be found in \cite{Munos}. In any case, the typical viewpoint is to simply use statistical estimates of these quantities when solving for approximate value functions.

Even in the case of model-based reinforcement learning, it is generally less common within the literature to see algorithms which adapt to, or measure the degree of uncertainty given in estimates of the state dynamics or objective functions. The most relevant works come from the Bayesian reinforcement learning community \cite{Duff,Poupart}, see also \cite{BRL-Survey}. Much of the work in the Bayesian reinforcement learning community focuses on \emph{partially observable Markov decision processes}, or on Gaussian processes. Recently more work has been done to model uncertainties in the context of transfer learning \cite{Killian}, and within the more general Gaussian process literature \cite{PILCO,Chowdhary,Jain2017}. 

The present work seeks to give one possible model for making control decisions which take into account the degree of uncertainty in the state dynamics. In particular, it extends the framework from \cite{PILCO}, as well as other similar frameworks from the Gaussian Process community \cite{Chowdhary,Jain2017}, and provides a rigorous analysis of the same.


The main goal of this work is to propose a framework for optimal control problems which dynamically gather information about state dynamics. We envision this as a toy model for many of the tasks in reinforcement learning. In particular, this provides a first step towards principled exploration in these types of control problems.

There is also a significant literature in the control community regarding control in uncertain settings. We outline a few of these fields in only the briefest of terms. Adaptive control seeks to simultaneously estimate system parameters and choose appropriate controls. Adaptive control is very similar to the standard framework of reinforcement learning \cite{sutton1992}. Robust control aims at constructing a controller which performs well under a variety of uncertainties arising in the  system dynamics. Robust optimal control has been widely studied both using a dynamic programming approach \cite{BasarBook} (along with the closely related $\mathcal{H}^\infty$ control) and using the Pontryagin maximum principle approach \cite{BoltyanskyPoznyak}, \cite{Warga91}, \cite{MP}, \cite{Bettiol-Khalil}. To achieve an ``optimal" reliable controller in presence of uncertainties, two kinds of  approaches are followed:  in the first case, one tries to optimize the worst case performance (within some set of possible system uncertainties). This leads to the classical min-max optimal control problems \cite{Vinter2005}. On the other hand, an alternative strategy for the selection of an optimization criterion involves minimizing the distance from a desired behavior. This second approach leads to the Riemann-Stieltjes optimal control problems \cite{Ross2015}. Such a  framework is similar to ours in that the optimization occurs outside of the averaging, but the focus is more on static parametric models and proving Pontryagin maximum principles. Lastly, dual optimal control seeks to model system uncertainty as a state variable \cite{Feldbaum,KlenskeHennig}. We do not attempt to make any exhaustive coverage of these fields here. We do remark that these fields tend to focus on more restrictive settings (such as linear problems), and on stabilization guarantees. The focus here is slightly different: we attempt to consider a very flexible model of both true state dynamics and the uncertainty associated with those dynamics. We then focus on an online setting where uncertainty is both tracked and decreases over time. In a sense what we do here is really an adaptive control setting with some type of modeled uncertainty. From another viewpoint, one could simply view our work as a rigorous mathematical study of a specific type of Bayesian reinforcement learning.

In this paper, we show the basic  properties of the proposed model, such as the existence of the optimal solution and the regularity of the related value function.  We then provide some \textit{local} relations that the value function has to satisfy. We remark that we do not seek to study algorithms for solutions of this model. This would be a more involved process, and would lie outside of the scope of this work. Such algorithms, for closely related problems, have been proposed in \cite{PILCO}. The goal instead is to consider the types of models that would be most effective in modeling uncertainty within control problems. In the future we plan to propose and test numerical methods for the solution of such problems.
%
%
%
\subsection{Proposed Model and Assumptions}

We consider controlled dynamics that are given by
\begin{displaymath}
  \dot x (s) = f(x(s),u(s))
\end{displaymath}
where $f: \R^n\times \R^m \to \R^n$ is a fixed, but generally \emph{unknown}, function belonging to a suitable set of functions $X$ (more details will be provided in the next section).

We suppose that an agent represents their knowledge of the environment (that is their present knowledge of $f$) as a time-varying probability measure defined on the space of functions $X$. That is, given any subset $E$ of $X$, the agent views the probability that $f \in E$ at time $t$ is given by $\pi(t) (E)$.


The agent's overall goal will be to minimize
\begin{displaymath}
  \min_u \int_0^\infty e^{-\lambda s}J(x(s),u(s)) \,ds.
\end{displaymath}
However, in light of the agent's lack of information, this task is approached using the following rules:
\begin{enumerate}
  \item[(i)]The agent makes decisions in a greedy fashion, optimizing their expected lifetime return given the present information.
  \item[(ii)] The agent passively gathers information over time about the environment. This could be expressed in many ways, but we will assume that learning occurs in a local neighborhood around the present state.
\end{enumerate}

We can then summarize this learning environment with the following problem statement. We suppose that $\tilde x(\cdot)$ represents the \emph{actual} state dynamics, that is
\begin{displaymath}
  \dot{\tilde{x}}(s) = \tilde f(\tilde x(s),\tilde u(s))\qquad s\in[0,t]
\end{displaymath}
where $\tilde u(\cdot)$ is the control that we have picked up to time $t$. At any time $t$ we will define the following minimization problem $(P_{t,\tilde{x}})$:
\begin{equation}
 V(\pi,\tilde x,t) = \min_u \, \mathcal{E}(\pi,\tilde x,t,u),
 \end{equation}
where $\mathcal{E}$ is defined by
 \begin{equation}\label{eqn:control-prob}
   \begin{aligned}
 \mathcal{E}(\pi,\tilde x,t,u)&:=\int_t^\infty \int_X e^{-\lambda s} J(x^f(s),u(s)) \, \pi(df) \,ds \\
  &s.t.\quad \dot x^f(s) = f(x^f(s),u(s)), \quad x^f(t) = \tilde{x}.
\end{aligned}
\end{equation}
This represents solving for an optimal open loop control, given current values for $\tilde x$ and $\pi$. Here $V$ is the value function, which depends on the current value of the state and $\pi$. The cost $\mathcal{E}(\pi,\tilde x,t,u)$ represents the value assuming that the dynamics are given by $f$ with a certain probability $\pi$ and using the information gathered by the previous states $\tilde{x}(t)$. This means that the agent looks forward in time, considering possible future rewards in each of the possible beliefs they have about the environment.
For fixed $t>0$, $x\in \mathbb{R}^{n}$ and a probability $\pi$, a minimizer for  problem (\ref{eqn:control-prob}) is a control $u^{*}$  such that
$$\mathcal{E}(\pi,\tilde x,t,u^{*})\leq \mathcal{E}(\pi,\tilde x,t,u)$$
for every $u$ admissible control (further details on the notion of admissible control will be provided).

This minimization problem will then be complemented by the actual state dynamics, which we write as
\begin{align}
\dot{\tilde{x}}(s) &= \tilde{f}(\tilde{x}(s),\tilde u(s))\qquad s\in[0,t]\\
\tilde{x}(0) &= x_{0},
\end{align}

Finally, one has to specify the manner in which $\pi$ changes over time. One could consider different frameworks for such a rule. One could consider, for example:
\begin{enumerate}
  \item[(i)] $\pi$ is updated using local information about the dynamics. This would represent an agent who can observe $f$ with some degree of accuracy near the current state. For example, one could consider a rule like
    \begin{equation}
      \begin{array}{ll}
  \pi(0)(E) = \pi(t)\bigg(\{h :h(y) = \tilde{f}(y) + e^{-\int_0^t \phi(|\tilde{x}(s)-\hat y|)\,ds}\\
   \cdot (g(y)-\tilde{f}(y)), g \in E\}\bigg).\label{eqn:example-dynamics}
 \end{array}
 \end{equation}
Here $\phi$ needs to be a compactly supported function which goes to infinity at zero (e.g. $\chi_{[0,1]}\cdot \frac{1-x}{x}$). Also, $\{\tilde{x}(s):\, s\in[0,t]\}$ is the collection of the previous states visited up to time $t$. $y = (\hat y,\bar y)$ is a vector combining the state and the control.  The function $\tilde u(s)$ are the actions that have been taken up to time $t$.
\item[(ii)] $\pi$ is given by some parametric representation, and one does a statistical estimation of these parameters using past observations of the state dynamics.
\item[(iii)] $\pi$ is given by some Bayesian problem: namely one computes the posterior distribution of state dynamics given some prior and some observations.

\end{enumerate}

Here our point of view will mostly focus on the first case. In particular, we will assume an \emph{absolute local learning hypothesis}, namely that
\begin{displaymath}
  \pi(t) (\{f : f(x,u) = \tilde f(x,u) \text{ for all } x \in B(\tilde x,\e)\;, u\in U\}) = 1.
\end{displaymath}
This assumption can be interpreted as follows: we assume that an agent learning in this framework can observe the state dynamics in some small region near their current position, and adjusts their belief $\pi$ of possible state dynamics accordingly.

Naturally, this type of hypothesis would be satisfied by dynamics of the form \eqref{eqn:example-dynamics}. Such a hypothesis would not necessarily hold for statistical estimation procedures. However, such an assumption does not seem too unrealistic.

This type of model represents an online learning environment where one passively learns about their environment and makes decisions regarding future actions given all of their current information. Our primary goal in this work will be to study well-posedness and optimality conditions for such a model. Future works will consider other aspects of this model, such as asymptotic learning, stability, approximation and algorithmic considerations.

Throughout the remainder of the paper, we use the following standing assumptions:
\begin{itemize}
\item[(H1)] Given a set of functions $X$ and $U \subset \mathbb{R}^{m}$, there exist constants $L>0$ and $C>0$ such that 
$$|f(x,u)-f(x',u')|\leq L\Big(|x-x'|+|u-u'|\Big)$$
and
$$|f(x,u)|\leq C$$
for every $(x,u),(x',u')\in \mathbb{R}^{n}\times U$, for every $f\in X$.
\item [(H2)] The mapping $(x,u)\mapsto J(x,u)$ is continuous and there exists a constant $L_{J}>0$ such that
$$|J(x,u)-J(x',u)|\leq L_{J}|x-x'|$$
for every $x,x'\in \mathbb{R}^{n}$ and $u\in U$.
\item[(H3)] The mapping $(x,u)\mapsto \tilde{f}(x,u)$ is an element of $X$ satisfying hypothesis (H1).
\end{itemize}

\subsection{Mathematical Preliminaries}

 For a fixed $x\in\mathbb{R}^{n}$ and $r>0$ we denote as $B(x,r)$ the ball in $\mathbb{R}^{n}$ centered at $x$ and with radius $r>0$.
Throughout the paper, we denote as $U\subset \mathbb{R}^{m}$ a compact subset and as $\mathcal{U}$ the set of the measurable function from $[0,\infty)$ taking values in $U$. Also, denote $\mathrm{r.p.m.}(U)$ as the set of Radon probability measure on $U$. We will refer to $\mathcal{U}$ as the set of original control functions. In general, it is well known \cite{WargaBook} that the set $\mathcal{U}$ does not have good compactness properties. For this reason, in this paper we will deal with relaxed controls, which will be defined by the set $\mathcal{R}$ of Borel measurable mappings from $[0,\infty)$ to $\mathrm{r.p.m.}(U)$. If $\sigma \in \mathcal{R}$, the related relaxed dynamics is 
\[
\dot{x}(t)=\int_{U}f(x(t),u)\sigma(t)(du)=:f(x(t),\sigma(t))\quad \mathrm{a.e.}\; t\in[0,\infty).
\]
In general, we can identify an element $u(\cdot)\in \mathcal{U}$ with the element $\delta_{u(\cdot)}\in \mathcal{R}$.  Given a sequence of elements $\{ \sigma_{j}\}\subset \mathcal{R}$, we say that $\sigma_{j}\rightarrow \sigma$ in the topology of $\mathcal{R}$ if 
$$\int_{0}^{\infty}\int_{U}\phi(t,r) \sigma_{j}(t)(dr)dt\rightarrow \int_{0}^{\infty}\int_{U}\phi(t,r) \sigma(t)(dr)dt$$
for every $  \phi\in L^{1}([0,\infty),C^{0}(U)).$ For a more detailed exposition on relaxed controls we refer to \cite{WargaBook}.

Denote by $C^{0}(\mathbb{R}^{n}\times U; \mathbb{R}^{n})$ the set of continuous function over $\mathbb{R}^{n}\times U$ and taking values in $\mathbb{R}^{n}$. Take $X\subset C^{0}(\mathbb{R}^{n}\times U; \mathbb{R}^{n})$ a set of equi-bounded and equi-Lipschitz functions (equivalently, $X$ is such that (H1) is satisfied). Then it follows from the Ascoli-Arzel\`a theorem that $X$ is compact. It then makes sense to define the set $\mathrm{r.p.m.}(X)$, that is the set of Radon probability measure of the compact set of functions $X$. In this paper we will consider mapping $\pi:[0,\infty)\rightarrow \mathrm{r.p.m.}(X)$ that will model the learning process of the system.

\section{Properties of the proposed model}
Define the value function
\begin{equation}\label{eqn:value-function}
W(t,s,x)=\inf_{u\in\mathcal{U}}\int_{s}^{\infty}\int_{X}e^{-\lambda\tau}J(x^{f}(\tau),u(\tau))\pi(t)(df)d\tau.
\end{equation}
We begin by demonstrating, as in the classical control case, that one can remove the dependence on $s$ from the value function.

\begin{proposition}\label{prop:s-dependence}
  The value function takes the form
  \begin{displaymath}
    W(t,s,x) = e^{-\lambda s} W(t,0,x). 
  \end{displaymath}
\end{proposition}
\begin{proof} In the integral which defines the value function $W(t,s,x)$,
 apply the change of variables $\tau=\xi+s$. This immediately implies
 \begin{equation}
   \begin{array}{ll}
     \int_{s}^{\infty}\int_{X}e^{-\lambda\tau}J(x^{f}(\tau),u(\tau))\pi(t)(df)d\tau\\[5pt]
     =e^{-\lambda s}\int_{0}^{\infty}\int_{X}e^{-\lambda\xi}J(x^{f}(\xi+s),u(\xi+s))\pi(t)(df)d\xi.
   \end{array}
 \end{equation}
Here, $x^{f}(s+\cdot)$ is the solution of the initial value problem $\dot{x}(s+\xi)=f(x(s+\xi), u(s+\xi))$, $x(s)=x$ for every $f\in X$. Rescaling the time variable, it is a straightforward matter to check that the previous initial value  is equivalent to $\dot{x}(\xi)=f(x(\xi), u(\xi))$, $x(0)=x$ for every $f\in X$. This completes the proof.
\end{proof}

It follows from the previous proposition that to obtain a complete characterization of $W(\cdot,\cdot,\cdot)$, it is enough to study the function
$$V(t,x):=W(t,0,x).$$

 Next we establish the existence of minimizers for the fixed-time problem. To do this, we first state two crucial propositions, which establish the continuity of the integral cost in the function describing the dynamics, and then the continuity (in the weak-* topology) of the integral functional with respect to the control.
 \begin{proposition} \label{prop:cont-map}Let us assume hypotheses (H1)-(H2) and that $\lambda>L$, where $L$ is the constant appearing in (H1). Then 
 \begin{itemize}
\item[i)] the mapping $(f,u)\mapsto J(x^{f}(\tilde\sigma)(t),u)$ is continuous on $X\times U$ for every $t\geq0$, $\tilde\sigma\in\mathcal{R}$;
\item[ii)] the mapping $f\mapsto J(x^{f}(\tilde\sigma)(t),\sigma)$ is continuous on $X$ for every $t\geq0$, $\tilde\sigma, \sigma\in\mathcal{R}$;
 \end{itemize}
 \end{proposition}
\begin{proof}
 Fix $\tilde\sigma \in \mathcal{R}$ and take  $f_{1}(\cdot,\cdot), f_{2}(\cdot,\cdot)\in X$ such that $||f_{1}(\cdot,u)-f_{2}(\cdot,u)||_{\infty}<\delta$ for every $u\in U$. Consider
 $x_{1}(\cdot)$ and $x_{2}(\cdot)$  solutions of
\begin{align}
\dot{x}(s) &= f_{i}(x(s), \tilde{\sigma}(s))\qquad s\in[0,\infty)\\
x(0) &= x_{0},
\end{align}
for $i=1,2$.
Using the uniform Lipschitz continuity of  the functions in $X$, we easily obtain
\[
\begin{array}{llll}
|x_{1}(t)-x_{2}(t)|\leq \int_{0}^{t}|f_{1}(x_{1}(s),u_{1})-f_{2}(x_{2}(s),u_{2})| ds \\[5pt]
\leq \int_{0}^{t}|f_{1}(x_{1}(s),\tilde{\sigma}(s))-f_{2}(x_{1}(s),\tilde{\sigma}(s))| ds \\[5pt]
+\int_{0}^{t}|f_{2}(x_{1}(s),\tilde{\sigma}(s))-f_{2}(x_{2}(s),\tilde{\sigma}(s))| ds \\[5pt]
\leq \delta t+ L \int_{0}^{t}|x_{1}(s)-x_{2}(s)|ds.
\end{array}
\]
It follows from Gr\"onwall's Lemma that
\begin{equation}\label{gronwall}
|x_{1}(t)-x_{2}(t)|\leq \delta\, t e^{Lt}.
\end{equation}
 Take $u_{1},u_{2}\in U$ such that $|u_{1}-u_{2}|<\delta$. In view of the hypothesis $(H2)$, we easily obtain the estimates
 $$ e^{-\lambda t}|J(x_{1}(t),u_{1})-J(x_{2}(t),u_{2})|\leq $$
 $$\leq e^{-\lambda t}\Big(|J(x_{1}(t),u_{1})-J(x_{2}(t),u_{1})|+|J(x_{2}(t),u_{1})-J(x_{2}(t),u_{2})|\Big)$$
 $$\leq \delta\, L_{J} t e^{-(\lambda-L) t}+e^{-\lambda t}\omega_{J}(\delta)$$
 for every $t\geq0$, where $\omega_{J}$ is the modulus of continuity of $J(\cdot,\cdot)$. This proves the  statement $i)$.
 
 The proof of $ii)$ follows the same initial steps, obtaining \eqref{gronwall} for  $x_{1}(\cdot),x_{2}(\cdot)$. Fix $\sigma \in \mathcal{R}$. An easy application of \eqref{gronwall} and of hypothesis $(H2)$ leads to
 $$ e^{-\lambda t}|J(x_{1}(\tilde{\sigma})(t),\sigma)-J(x_{2}(\tilde{\sigma})(t), \sigma)|\leq \delta\, L_{J} t e^{-(\lambda-L) t}$$
 for every $t\geq0$ and $\sigma, \tilde{\sigma}\in\mathcal{R}$. This concludes the proof.
 \end{proof}

\begin{proposition}\label{prop:continuity-functional}
  Let us assume hypotheses (H1)-(H2) and that $\lambda>L$, where $L$ is the constant appearing in (H1). Then the functional $G: \mathcal{R} \to \R$, given by
  \begin{displaymath}
   G(\sigma) = \int_s^\infty e^{-\lambda \tau} J(x^f(\tau),\sigma(\tau)) \pi(df) \,d\tau
 \end{displaymath}
 is continuous (in the topology of $\mathcal{R}$).
\end{proposition}
\begin{proof}
  Suppose that $\sigma_n$ converges to $\sigma$ (in the sense of generalized controls). Let $x_n^f(t)$ be the solution of
  \begin{displaymath}
    \begin{array}{ll}
    x_n^f(t) = x_0 + \int_0^t f(x_n^f(s),\sigma_n(s)) \,ds \\
    = x_{0}+\int_0^t \int_U f(x_n(s),r) \sigma_{n}(s) (dr) \,ds,
  \end{array}
  \end{displaymath}
  and
  \begin{displaymath}
    x^f(t) = x_0 + \int_0^t \int_U f( x^f(s),r) \sigma(s) (dr) \,ds.
  \end{displaymath}
  Using that $\sigma_n \to \sigma$ in the sense of relaxed controls, that $f$ is Lipschitz in $x$ and that $f( x^f(s),r)$ is continuous in $s$, one can use Gr\"onwall's inequality once again to obtain that 
  \begin{equation}\label{eqn:x-conv}
    |x_n^f(t)-x^f(t)| \leq C(n)e^{Lt},
  \end{equation}
  where $C(n)$ is a positive constant approaching zero as $n \to \infty$.

  We next estimate the difference between (abusing notation), $\mathcal{E}(\sigma_n)$ and $\mathcal{E}(\bar \sigma)$; in other words the main task is to estimate
  \begin{displaymath}
    \lim_{n} \!\int_0^\infty\!\int_X\! e^{-\lambda s} \left( J(x_n^f(s),\sigma_n(s)) - J( x^f(s),\sigma(s) ) \right)\pi(t) (df) ds.
  \end{displaymath}
  By adding and subtracting $J(x^f(s),\sigma_n(s))$, and letting $\epsilon = \lambda - L > 0$, we then need to estimate
  \begin{equation}
  \begin{array}{ll}
  \leq \lim_{n \to \infty} \int_0^\infty e^{-(\epsilon/2)s} \int_X e^{-(\lambda-\epsilon/2) s}\bigg(L_{J} |x_n^f(s) - x^f(s)| \\
  + J(x^f(s),\sigma_n (s)) - J(x^f(s),\sigma(s)) \bigg)\pi(t)(df) ds,
  \label{eqn:est1}
\end{array}
\end{equation}
Using the dominated convergence theorem we can move the limit to the inside of the integral in $s$. Then using Arzela's bounded convergence theorem (namely that the dominated convergence theorem holds for Riemann integrals when the limit is also Riemann integrable), we can pass the limit in $n$ to the intermost integral. The first term goes to zero using \eqref{eqn:x-conv}. For what concern the second term let us write explicitly:
$$\int_0^\infty e^{-\lambda s} \int_X  \int_{U} J( x^f(s),r)(\sigma_{n}(s)-\sigma(s))(dr)\,\pi(df) ds.$$
It follows from Proposition \ref{prop:cont-map}, part $i$, that the mapping $(f,u)\mapsto e^{-(L + \epsilon/2) s} J(\bar x^{f}(s),u)$ is continuous on $X\times U$ for each $s>0$. A simple application of Fubini's Theorem yields
$$\int_0^\infty e^{-\lambda s}   \int_{U} \int_X J( x^f(s),r)\pi(df)\,(\sigma_{n}(s)-\sigma(s))(dr)\, ds,$$
which goes to $0$ for $n\rightarrow \infty$ since $\sigma_{n}\rightarrow \sigma$ in $\mathcal{R}$. 
This concludes the proof.

\end{proof}

The following two corollaries are immediate consequences of the previous proposition:

\begin{corollary}\label{cor:existence}
  Under the assumptions of Proposition \ref{prop:continuity-functional}, there exists a minimizer of the variational problem
  \begin{displaymath}
    \min_{\sigma \in \mathcal{R}} \int_0^\infty \int_X e^{-\lambda \tau} J(x^f(\tau),\sigma(\tau)) \pi(t)(df) \,d\tau.
  \end{displaymath}
\end{corollary}

\begin{proof}\label{cor:same_val}
  First, since $J$ is continuous, and the integral cost has a decaying exponential weight, it is clear that the infimum is finite, and so we can select a minimizing sequence $\sigma_n \in \mathcal{R}$. By taking a subsequence, we will have that $\sigma_n$ converges (in the sense of generalized controls, that is weakly star) to $\bar\sigma$. Proposition \ref{prop:continuity-functional} then establishes the desired result.
\end{proof}

\begin{corollary}
  The value functions associated with standard and generalized controls are the same, meaning that
  \begin{displaymath}
   V(t,x) = \min_{\sigma \in \mathcal{R}} \int_0^\infty \int_X e^{-\lambda \tau} J(x^f(\tau),\sigma(\tau)) \pi(t)(df) \,d\tau.
  \end{displaymath}
\end{corollary}

\begin{proof}
  Given a minimizing relaxed control $\bar \sigma \in \mathcal{R}$, we can approximate it (in the topology of $\mathcal{R}$) using a sequence in $\mathcal{U}$ (see \cite{WargaBook}). The result then follows using again Proposition \ref{prop:continuity-functional}. 
\end{proof}

We now give some simple regularity properties of the value function. In what follows we will assume that the mapping $t\mapsto \pi(t)$ is either lower semicontinuous w.r.t. the weak-* topology, namely
\begin{equation}\label{eqn:pi-low_cont}
\int_{X}c(f)\pi(t)(df)\leq \liminf_{s\rightarrow t}\int_{X}c(f)\pi(s)(df) 
\end{equation}
for every $ c\in C^{0}(X)$,
or Lipschitz continuous w.r.t. the weak-* topology, namely that there exists $L_{\pi}>0$ such that
\begin{equation}\label{eqn:pi-Lip}
\Big| \int_{X}c(f)\pi(t)(df)-\int_{X}c(f)\pi(s)(df)\Big|\leq L_{\pi}|t-s|
\end{equation}
for every $ c\in C^{0}(X)$.
Then the next result follows:

\begin{proposition}\label{prop:Lipschitz}
  Let us assume hypotheses (H1)-(H2) and that $\lambda>L$, where $L$ is the constant appearing in (H1).  Then the mapping $x\mapsto V(t,x)$ is Lipschitz continuous for every $t\geq 0$.
  Furthermore:
  \begin{itemize}
 \item[$i)$] if $\pi$ is lower semicontinuous in $t$ (namely \eqref{eqn:pi-low_cont}), then the mapping $t\mapsto V(t,x)$ is lower semicontinuous for every $x\in \mathbb{R}^{n}$.
  \item[$ii)$] if $\pi$ is Lipschitz continuous in $t$ (namely \eqref{eqn:pi-Lip}), then the mapping $t\mapsto V(t,x)$ is Lipschitz continuous for every $x\in \mathbb{R}^{n}$.
   \end{itemize}
\end{proposition}
\begin{proof}
We first prove the Lipschitz continuity regularity w.r.t. $x$. Fix $t>0$. In view of Corollary \ref{cor:existence}, there exists $\sigma^{*}\in \mathcal{R}$ such that
$$V(t,x)=\int_{0}^{\infty}\int_{X}e^{-\lambda s}J(x^{f}(s,x),\sigma^{*}(s))\pi(t)(df)ds.$$
 Then it easily follows that
 \begin{equation}
   \begin{array}{llll}
|V(t,x)-V(t,y)|\leq  \\[5pt]
\int\limits_{0}^{\infty}e^{-\lambda s}\int\limits_{X}\Big|J(x^{f}(s,x),\sigma^{*}(s))-J(x^{f}(s,y),\sigma^{*}(s))\Big|\pi(t)(df)ds \\[5pt]
\leq \int_{0}^{\infty}e^{-\lambda s}\int_{X}L_{J}|x^{f}(s,x)-x^{f}(s,y)|\pi(t)(df)ds\\[5pt]
 \leq L_{J}\int_{0}^{\infty} e^{-(\lambda-L)s}|x-y|ds=C|x-y|,
\end{array}
\end{equation}
where we have used, respectively, the Lipschitz continuity of $J(\cdot,u)$, Gr\"onwall's lemma and the hypothesis $\lambda>L$.

Now we concentrate on the regularity w.r.t. $t$. Let us assume assumption \eqref{eqn:pi-low_cont} and let us fix $\epsilon=\lambda-L$. Then, in view of Proposition \ref{prop:cont-map}, $ii)$, the mapping \\$f\mapsto  e^{-(\lambda -\epsilon/2)}J(x^{f}(s),\sigma(s))$ is continuous for each $\sigma\in \mathcal{R}$. Fix $h>0$ and call $\sigma^{*}_{\tau}(\cdot)\in \mathcal{R}$ the optimal control such that
$$V(\tau,x)=\int_{0}^{\infty}\int_{X}e^{-\lambda s}J(x^{f}(s,x),\sigma_{\tau}^{*}(s))\pi(\tau)(df)ds.$$ 
It now follows, using the Fatou's Lemma and \eqref{eqn:pi-low_cont}, that
$$\liminf_{\tau\rightarrow t} V(\tau,x)=$$
$$\liminf_{\tau\rightarrow t} \int_{0}^{\infty}\int_{X}e^{-\lambda s}J(x^{f}(s,x),\sigma^{*}_{\tau}(s))\pi(\tau)(df)ds\geq$$
$$\geq \inf_{t-h<r<t+h}\liminf_{\tau\rightarrow t} \int_{0}^{\infty}\int_{X}e^{-\lambda s}J(x^{f}(s,x),\sigma^{*}_{r}(s))\pi(\tau)(df)ds$$
$$ \geq \inf_{t-h<r<t+h} \int_{0}^{\infty}\liminf_{\tau\rightarrow t}\int_{X}e^{-\lambda s}J(x^{f}(s,x),\sigma^{*}_{r}(s))\pi(\tau)(df)ds$$
$$ \geq \inf_{t-h<r<t+h}\int_{0}^{\infty}\int_{X}e^{-\lambda s}J(x^{f}(s,x),\sigma^{*}_{r}(s))\pi(t)(df)ds\geq$$
$$\geq V(t,x)$$

In particular, since the previous relations hold true for every $h>0$, we obtain that
\begin{equation}
 V(t,x)\leq \liminf_{\tau\rightarrow t}V(\tau,x) 
\end{equation}
for every $x\in\mathbb{R}^{n}$, which concludes the proof of $i)$.

We now show property $ii)$. Let us now assume hypothesis \eqref{eqn:pi-Lip}). Fix $x$,  and take $t,\tau>0$. Then, in view of an application of Proposition \ref{prop:cont-map}, $ii)$, to the mapping $f\mapsto e^{-(\lambda-\epsilon/2) s} J( x^{f}(s),\sigma^{*}(s))$  and in view of the Lipschitz regularity of $t\mapsto \pi(t)$, we obtain that
\begin{equation}
  \begin{array}{lll}
|V(t,x)-V(\tau,x)|\\[5pt]
\leq  \Big|\int_{0}^{\infty}e^{-\lambda s}\int_{X}J(x^{f}(s),u^{*}(s))(\pi(t)-\pi(\tau))(df)\,ds\Big| \\[5pt]
\leq \int_{0}^{\infty}e^{-(\epsilon/2) s}L_{\pi}|t-\tau|ds=C|t-\tau|.
\end{array}
\end{equation}
This proves relation $ii)$  and concludes the proof.
\end{proof}
%
%
%
%
%
%
%
We recall the \emph{absolute local learning} assumption around a point $\tilde{x}$ given in the introduction; namely that
\begin{equation}\label{eqn:abs-learn}
  \pi(t)(\{f : f(x,u) = \tilde f(x,u) \text{ for all } x \in B(\tilde x,\e)\;, u\in U\}) = 1.
  \end{equation}
We remind the reader that this assumption can be interpreted as follows: we assume that an agent learning in this framework can observe the state dynamics in some small region near their current position, and adjusts their belief $\pi$ of possible state dynamics accordingly.

We now state our dynamic programming principle.

\begin{proposition}\label{prop:DPP} Assume assumptions $(H1)$,$(H3)$.
  Assume also that $\pi$ satisfies the absolute local learning assumption \eqref{eqn:abs-learn} around a point $\tilde{x}$. Then for any $h$ satisfying 
 \begin{equation}
   0 < h < \sup_{f \in supp(\pi)}|f|_\infty \e
 \end{equation}
 we have the following dynamic programming principle locally around $\tilde{x}$:

\begin{equation}
  \begin{aligned}
&V(t,\tilde{x}) = \\
&\inf_{u\in \mathcal{U}} \Big(\int_{0}^{h} J(x^{\tilde{f}}(s),u(s))e^{-\lambda s}ds  +e^{-\lambda h}V(t,\tilde{x}(h,u))\Big) \label{eqn:DPP}
\end{aligned}
\end{equation}
\end{proposition}
where $\tilde{x}(s,u)$ has to be regarded as the solution of 
$$\dot{x}(s)=\tilde{f}(x(s),u(s)),\quad x(0)=\tilde{x},\quad s\in[0,h],\quad u\in\mathcal{U}.$$
\begin{proof}[Proof of Proposition \ref{prop:DPP}] 
By the definition of $V(t,\tilde{x})$, for every $\delta>0$, there exists $u^{\delta}\in \mathcal{U}$ such that 
\begin{equation}
  \begin{array}{ll}
    \int_{0}^{h} J(x_\delta^{\tilde{f}}(s),u^{\delta}(s))e^{-\lambda s}ds\\[5pt]
    +\int_{h}^{\infty}\int_{X}J(x_{\delta}^{f}(s),u^{\delta}(s))e^{-\lambda s}\pi(t)(df)ds \leq V(t,\tilde{x}) + \delta.
  \end{array}
  \end{equation}
  Here, $x_{\delta}^{\tilde{f}}(\cdot)$ is the solution of the problem
  \begin{align}
\dot{x}(s)= & \tilde{f}(x(s),\tilde u_{\delta}(s))\qquad s\in[0,h]\\
\tilde{x}(0)=& x_{0},
\end{align}
while $x_{\delta}^{f}(\cdot)$ solves 
  \begin{align}
\dot{x}(s)= & f(x(s),\tilde u_{\delta}(s))\qquad s\in[h,\infty]\\
x(t)=& \tilde{x}(h).
\end{align}
Taking the infimum on the left hand side over the controls varying on the interval $[h,\infty)$, we obtain
$$ \int_{0}^{h} J(x_\delta^{\tilde{f}}(s),u^{\delta}(s))e^{-\lambda s}ds+W(t,h,\tilde{x}(h)) \leq V(t,\tilde{x}) + \delta,$$
which, in view of Proposition \ref{prop:s-dependence}, can be written as
$$\int_{0}^{h} J(x_\delta^{\tilde{f}}(s),u^{\delta}(s))e^{-\lambda s}ds+e^{-\lambda h}V(t,\tilde{x}(h)) \leq V(t,\tilde{x}) + \delta.$$
Taking now the infimum over the control varying on the time interval $[0,h]$ and letting $\delta\rightarrow 0$ we obtain
\begin{align}
\inf_{u\in\mathcal{U}}&\Big\{\int_{0}^{h} J(x^{\tilde{f}}(s),u(s))e^{-\lambda s}ds+\\
&+e^{-\lambda h}V(t,\tilde{x}(h,u)) \Big\}\leq V(t,\tilde{x}).
\end{align}

We now aim at proving the reverse inequality. From the definition  $V(t,\tilde{x})$, it easily follows
 \begin{align}
&V(t,\tilde{x})\leq \int_{0}^{h}e^{-\lambda s}J(x^{\tilde{f}}(s), u(s))ds \\
&+ \int_{h}^\infty \int_X  e^{-\lambda s}J(x^f(s),u(s)) \pi(t)(df)ds
\end{align}
for every $u \in \mathcal{U}$. Arguing as in the previous step of the proof, we can take the infimum over controls varying on the time interval $[h,\infty)$ and use the relation $W(t,h,\tilde{x}(h,u))=e^{-\lambda h}V(t,\tilde{x}(h,u)$. This in particular provides the inequality
$$V(t,\tilde{x})\leq \int_{0}^{h}e^{-\lambda s}J(x^{\tilde{f}}(s), u(s))ds +e^{-\lambda h}V(t,\tilde{x}(h,u)).$$
Taking now the infimum over controls varying on the interval $[0,h]$, the inequality
$$V(t,\tilde{x})\leq\inf_{u\in\mathcal{U}}\Big\{\int_{0}^{h} J(x^{\tilde{f}}(s),u(s))e^{-\lambda s}ds+e^{-\lambda h}V(t,\tilde{x}(h,u))\Big\}$$
easily follows. This completes the proof.

\end{proof}

\section*{Differential equation for optimal trajectories}

In light of the dynamic programming principle in Proposition \ref{prop:DPP}, one can prove the following Hamilton-Jacobi equation using standard techniques:

\begin{theorem}
  Assume hypotheses (H2)-(H3). Suppose that $\pi(t)$ satisfies relation \eqref{eqn:abs-learn} in $B(x,r)$. Then, for $y \in B(x,r)$, the value function satisfies the Hamilton-Jacobi equation
  \begin{equation}
    \lambda V(t,y) = \inf_{u \in U} \left( J(y,u) + \nabla V(t,y) \cdot \tilde f( y,u) \right)
    \label{eqn:HJB-in-x}
  \end{equation}
in the sense of viscosity solutions.
\end{theorem}

\textbf{Remark:} The theory of viscosity solution for the equation \eqref{eqn:HJB-in-x} is well known (see, e.g. \cite{Bardi-Capuzzo}). For example, if one fixes the boundary values of $V$ on $\partial B(x,r)$, then there exists a unique viscosity solution of \eqref{eqn:HJB-in-x}. Of course in this setting those boundary values are not a priori known, and may be difficult to obtain. Even so, the Hamilton-Jacobi equation \eqref{eqn:HJB-in-x} provides important local information about the value function and the optimal control.

%
\vspace{0.1in}

The previous remark motivates the importance of providing a relation describing how the value function $V(t,x)$ evolves w.r.t. $t\in[0,\infty)$.

\begin{definition}
Given a scalar valued, lower semicontinuous function $g(\cdot)$, the strict sub-differential of $g$ at $t$ is defined as the set 
$$\hat{\partial}_{t}g(t)=\Big\{\xi\in\mathbb{R}:\,\limsup_{s\rightarrow t}\frac{\xi \cdot (s-t)-(g(s)-g(t))}{|s-t|}\leq 0\Big\}$$
\end{definition}

Now we demonstrate that one can establish a differential relation in $t$:

\begin{theorem}\label{thm:inclusion}
  Let us assume hypotheses (H1)-(H2) and that $\lambda>L$, where $L$ is the constant appearing in (H1). Furthermore, suppose that the mapping 
  $t\mapsto \pi(t)$ satisfies relation \eqref{eqn:pi-low_cont}. Then
  \begin{equation}\label{eqn:partial-t-exists}
    \hat{\partial}_{t} \Big(\int_0^\infty \int_X J(x^f(\tau),\sigma^*(\tau)) e^{-\lambda \tau} d\pi(t)d\tau\Big)
  \end{equation}
  is well-defined and satisfies the
  \begin{equation}\label{eqn:upper-t}
     \hat{\partial}_{t}V(t,x) \subseteq \hat{\partial}_{t} \Big(\int_0^\infty \int_X J(x^f(\tau),\sigma^*(\tau))e^{-\lambda \tau} d\pi(t)(df) \,d\tau\Big).
  \end{equation}
 \label{thm:t-derivative}
\end{theorem}

\begin{proof}
Notice that, in view of the hypothesis \eqref{eqn:pi-low_cont}, the function 
$$t\mapsto\Big(\int_0^\infty \int_X J(x^f(\tau),\sigma(\tau)) e^{-\lambda \tau} d\pi(t)d\tau\Big)$$
is lower semicontinuous for every $\sigma\in\mathcal{R}$. So its strict sub-differential is always well-defined. Furthermore, it follows from Proposition \ref{prop:Lipschitz}, $i)$, that the mapping
$t\mapsto V(t,x)$ is lower semicontinuous for every $x\in\mathbb{R}^{n}$.
By the definition of $V$ and the existence of an optimal generalized control $\sigma^{*}$, it easily follows that
\begin{displaymath}
  V(t,x) = \int_0^\infty \int_X J(x^f(\tau),\sigma^*(\tau))d\pi(t)(df) \,d\tau.
\end{displaymath}
and
\begin{displaymath}
  V(s,x) \leq \int_0^\infty \int_X J(x^f(\tau),\sigma^*(\tau))d\pi(s)(df) \,d\tau.
\end{displaymath}
for any $s$ close to $t$.
Fix $\xi\in \hat{\partial}_{t}V(t,x)$. From the previous inequalities, we obtain that
\begin{align}
  &\frac{\xi\cdot (s-t)-(V(s,x) - V(t,x))}{|s-t|} \geq \\
  &\frac{\xi\cdot (s-t)-\int_0^\infty \int_X J(x^f(\tau),\sigma^*(\tau)) e^{-\lambda \tau} (\pi(s)- \pi(t)) (df) \,d\tau}{|s-t|} .
\end{align}
Taking the $\limsup$ on both sides for $s\rightarrow t$, we achieve relation \eqref{eqn:upper-t}. This concludes the proof.
\end{proof}

\textbf{Remarks:} \begin{itemize}
\item[$1)$]The previous characterizations of the value function $V(t,x)$ are well defined even when the mapping $t\mapsto V(t,x)$ is merely lower semicontinuous.
In particular, such a feature permits  to characterize the value function even when there is a \textit{discontinuity} in the learning process, that is when the updated measurements from the environment affect a drastic (i.e. discontinuous) change to the mapping $\pi(t)$. 
\item[$2$) ]Another interesting implication of Theorem \ref{thm:inclusion} is that, if the mapping $$t\mapsto \Big(\int_0^\infty \int_X J(x^f(\tau),\sigma(\tau)) e^{-\lambda \tau} d\pi(t)d\tau\Big)$$ is differentiable for each $\sigma\in\mathcal{R}$, then its strict sub-differential is a singleton. This fact in particular implies that the mapping $t\mapsto V(t,x)$ is differentiable, providing a further regularity result for the value function.
\end{itemize}

\section*{Conclusion}

In this work we have considered control problems with uncertainty in the system dynamics. This situation is closely related to a variety of models in reinforcement learning and robust control. In particular, we have rigorously proven a dynamic programming principles and differential equations satisfied by the value function in such systems. 
We hope that these rigorous results can provide an impetus for more precise analysis of these types of models in control and learning.

\section*{References}

\bibliography{Reinforce}

\end{document}